\documentclass[12pt]{article}
\usepackage[utf8]{inputenc}
\usepackage{natbib}

\usepackage{fancyhdr}
\usepackage{extramarks}
\usepackage{amsmath}
\usepackage{amsthm}
\usepackage{amsfonts}
\usepackage{bbm}
\usepackage{bm}
\usepackage{siunitx}
\usepackage{adjustbox}
\usepackage{tikz}
\usepackage[plain]{algorithm} 
\usepackage{algpseudocode}
\usetikzlibrary{automata,positioning}
\usepackage{listings}
\usepackage{ulem}
\usepackage{booktabs,multirow}
\newtheorem{theorem}{Theorem}

\theoremstyle{definition}
\newtheorem{definition}{Definition}
\newtheorem{proposition}{Proposition}

\begin{document}

\begin{center}
{\Large
	{\sc  Choice of mixture Poisson models based on Extreme value theory}
}
\medskip

 $\text{Samuel Valiquette}^{a,b,c,d}$ \& $\text{Frédéric Mortier}^{a}$ \& $\text{Jean Peyhardi}^b$ \& $\text{Gwladys Toulemonde}^{b,c}$  
\medskip

{\it
$^a$ CIRAD, UPR Forêts et Sociétés, F-34398 Montpellier, France.
Forêts et Sociétés, Univ Montpellier, CIRAD, Montpellier, France.\\
$^b$ IMAG, CNRS, Université de Montpellier, 34090, Montpellier, France.\\
$^c$ LEMON, Inria, 34095, Montpellier, France \\
$^d$ Université de Sherbrooke, Département de mathématiques, Sherbrooke, Canada, J1K 2R1\\

}
\end{center}
\medskip

{\bf Abstract.} 
Count data are omnipresent in many applied fields, often with overdispersion due to an excess of zeroes or extreme values.  With mixtures of Poisson distributions representing an elegant and appealing modelling strategy, we focus here on the challenging problem of identifying a suitable mixing distribution and study how extreme value theory can be used.  We propose an original strategy to select the most appropriate candidate among three categories: Fr\'echet, Gumbel and pseudo-Gumbel. Such an approach is presented with the aid of a decision tree and evaluated with numerical simulations.


\section{Introduction}
\label{Intro}
Count data are classically observed in many applied fields such as in actuarial science when evaluating risk and the pricing of insurance contracts \citep[e.g.,][]{Claims}, in genetics to model the number of genes involved in phenotype variability \citep[e.g.,][]{Genes} or in ecology to model species abundance  \citep[e.g.,][]{Abundance}.
While Poisson models and regression are well established choices for these type of data, they are not suitable for overdispersed data, which typically occur with an excess of zeroes or extreme large values.  
To overcome such limitations the use of Poisson mixture models has been proposed.
This assumes the Poisson's intensity is no longer an unknown fixed value, but a  positive random variable. 
Mixture approach induces overdispersed distributions with more zeroes and high values compared to the classical Poisson model \citep{Shaked}. 
A variety of mixture distributions has been already proposed \citep{Karlis}.
Classical examples include the gamma distribution \citep{Greenwood}, the lognormal \citep{Bulmer} or the Bernoulli \citep{Lambert}.
From a general point of view, any distribution with a non negative support, finite or not, can be a potential candidate as a mixing distributions.
Some tests exist to verify whether data are overdispersed \citep{YHA} or if they come from a mixed Poisson distribution \citep{Carriere}.
These tests justify the use of Poisson mixtures, but do not make claims on what type of mixing distribution should be selected.
To our knowledge, there are no studies that propose a solution to this problem.
This paper aims to propose a new strategy to select an appropriate and efficient mixing distribution family.
\medskip

Usually one may choose to fit some predetermined Poisson mixtures and keep the best model based on some criteria. 
However the following example shows the difficulties in the mixing distribution choice. 
A sample of size $n=500$ has been simulated from a Poisson-Beta type II with parameters equal to $a = 1$ and $b = 2.2$.
This choice ensures finite expectation and variance for the mixed Poisson distribution and are equal to $0.83$ and $8.47$ respectively.
Such a distribution has been used to model accident proneness \citep{Holla}.
First and as expected, the use of a simple Poisson model fails to properly fit data (see Figure \ref{fig:example}).
In particular, it does not capture well the high frequency of zeroes and large values once its parameter is estimated ($\hat{\lambda}= 1.02$).

\begin{figure}[H]
    \centering
    \includegraphics[width=\textwidth]{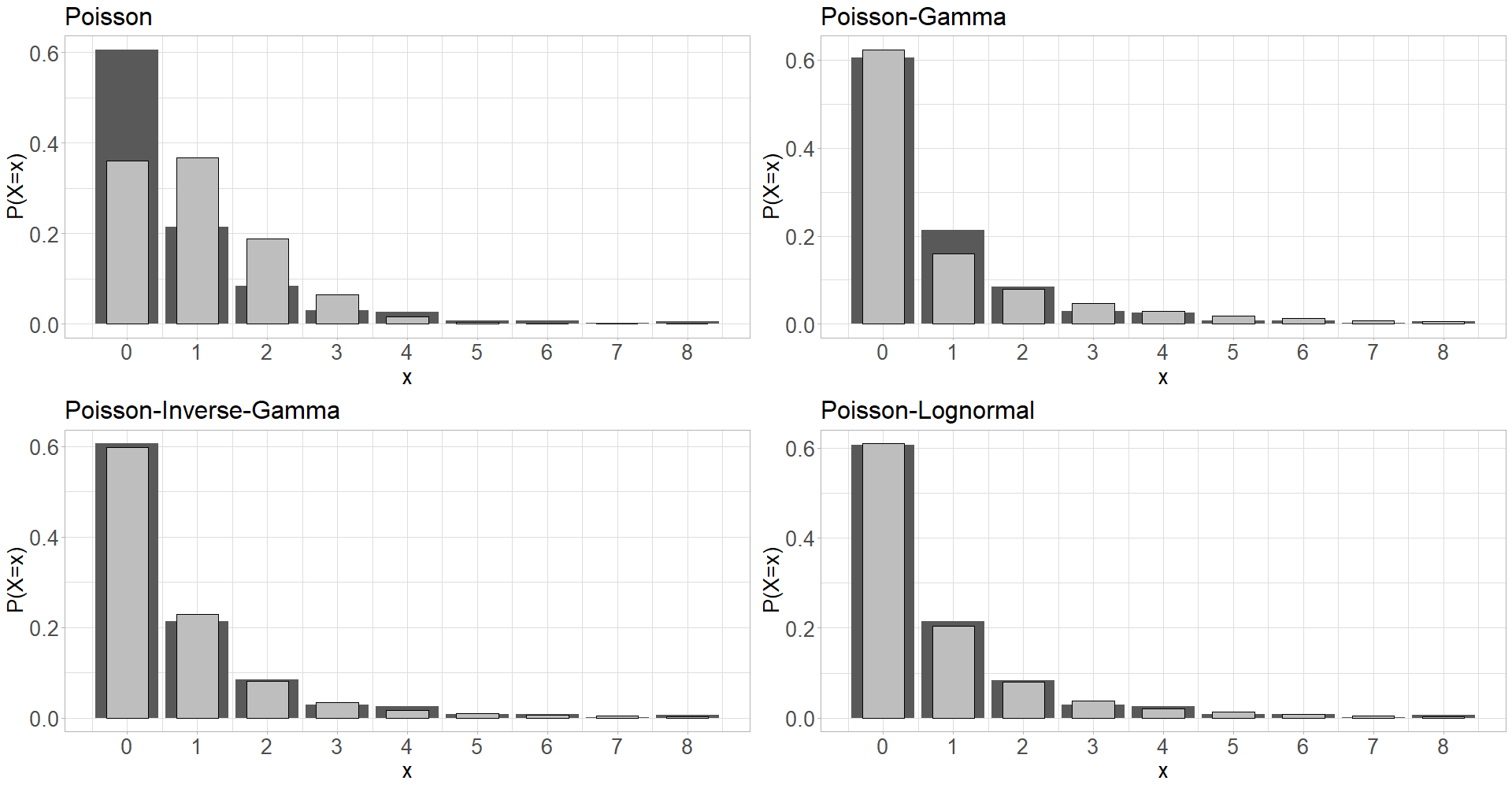}
    \caption{Fitted Poisson, Poisson-Gamma, Poisson-Inverse-Gamma and Poisson-Lognormal models to data simulated from a Poisson-Beta type II distribution with $a = 1$ and $b = 2.2$  compared to its empirical distribution (Dark Gray) }
    \label{fig:example}
\end{figure}

\pagebreak

Alternative latent distributions on $\lambda$ have been used for this simulation, either the gamma or lognormal as classically proposed in ecology \citep{JSDM} or the inverse-gamma applied in actuarial science for liability insurance claims \citep{Tzougas}.
Inference is performed in a Bayesian framework (details in Section \ref{Section_Mixture}.2).
While the Poisson-gamma, \textit{i.e.} negative binomial, or the Poisson-lognormal are popular choices, the Poisson-inverse-gamma is privileged for this example.
Indeed, the \textit{posterior} model probabilities of the negative binomial and Poisson-lognormal are respectively $0$\% and $38.8$\% compared to $61.2$\% for the Poisson-inverse-gamma.
The latter mixture distinguishes itself even if the three models behave similarly with regards to Figure \ref{fig:example}.\medskip

The reason that the inverse-gamma is privileged in this example is due to its tail behavior which is similar to the beta type II distribution.
We establish in this paper that an adequate choice for the distribution on $\lambda$ can be made by analysing such a property.
This paper is organized as follow. 
Section \ref{Section_Mixture} presents a classification of various Poisson mixtures based on their tail behavior using extreme value theory. 
Using these new classifications, we construct in Section \ref{Section_Strategy} a strategy to choose a family of distributions for $\lambda$. 
This strategy is presented in the form of a decision tree where each step leads to an adequate category of Poisson mixtures.
Simulations for each section are also presented to attest the relevance and usefulness of such an approach.

\section{Poisson mixture tail behavior}\label{Section_Mixture}

In this section, we present the fundamental results in extreme value theory and the restrictions when it comes to discrete distributions. 
Following this, we present three categories of Poisson mixtures that characterise different tail behaviors and conclude with simulations to assess the interest of selecting such a category.

\subsection{Theoretical foundations and Poisson mixtures categories}

The tail behavior of a distribution can be studied using extreme value theory.
Such a statistical approach analyses how the maximum of a distribution $F$ stabilizes.
The theory says that $F$ belongs to a max domain of attraction if it exists two normalizing sequences $a_n > 0$ and $b_n$ such that $F^n(a_n x - b_n)$ converges to a non-degenerate distribution when $n \to \infty$  \citep{Resnick}.
There are three possible domains of attraction named Weibull, Gumbel and Fréchet. 
These domains describe the asymptotic tail behavior of $F$ and correspond respectively to finite, exponential and heavy tailed distributions. 
In the sequel, they will be denoted by $\mathcal{D}_-$, $\mathcal{D}_0$ and $\mathcal{D}_+$, and we will write this property by $F \in \mathcal{D}$ where $\mathcal{D}$ is one of the three domains.
In the following, we assume that the mixing distribution $\lambda$ is supported on $\mathbb{R}_+$, then only $\mathcal{D}_0$ and $\mathcal{D}_+$ are considered.
Usual continuous distributions belong to a domain, this is not the case for discrete random variables.
Indeed, a necessary condition for a discrete distribution $F$ to be in a domain of attraction is the property of long tailed \citep{Anderson} defined by:

\begin{equation}
    \lim_{n \to \infty} \frac{1-F(n+1)}{1-F(n)} = 1.
\end{equation}
In particular well known discrete distributions do not satisfy this property such as among others Poisson, negative binomial or yet geometric distributions.
Even though the latter two are Poisson mixtures with a gamma distributed mixing parameter, which belongs to $\mathcal{D}_0$, the domain of attraction is not carried over by the  mixture distribution. 
However \citet{Anderson} and \citet{Shimura} showed that if a discrete distribution verifies 
\begin{equation}
    \lim_{n \to \infty} \frac{1-F(n+1)}{1-F(n)} = L \in (0,1)
\end{equation}
then $F$ is, in a sense, 'close' to the Gumbel domain.
More precisely, \citet{Shimura} showed that property (2) implies that $F$ is the discretization of a unique continuous distribution belonging to $\mathcal{D}_0$.
On the other hand, \cite{Anderson} showed that there is a sequence $b_n$ such that $F^n(x+b_n)$ has infimum and supremum limits bounded by two different Gumbel distributions, implying that $F$ is not far from this domain.\bigskip 

Because Poisson mixture distributions are discrete distributions, they are constrained to the long tailed property in order to have a domain. 
Otherwise, they may be close to the Gumbel domain.
But Poisson mixtures are uniquely identifiable by the distribution on $\lambda$ \citep{Feller}, this means that its tail behavior is dependent of the latter.
Therefore, we need to understand conditions on $\lambda$ that allow Poisson mixture distribution to inherit a domain or to be close to the Gumbel one.
\cite{Perline} established some conditions for such preservation.
From this point forward, we note by $F$ and $f$ the cumulative
distribution function (cdf) and the probability density function (pdf) for $\lambda$ and $F_M$ the cdf of the resulting Poisson mixture.
Firstly, if $F \in \mathcal{D}_+$ and is such that $\lim_{x \to \infty} \frac{x f(x)}{1-F(x)} = \alpha$ for some $\alpha > 0$ (1\textsuperscript{st} Von Mises condition), then $F_M \in \mathcal{D}_+$.
Secondly, if $F\in \mathcal{D}_0$, $\lim_{x \to \infty} \frac{d}{dx}\left[ \frac{1-F(x)}{f(x)}  \right] = 0$ (3\textsuperscript{rd} Von Mises condition) and $\frac{f(x)}{1-F(x)} = o(x^{-\delta})$ as $x \to \infty$ for some $\delta \geq \frac{1}{2}$, then $F_M \in \mathcal{D}_0$.\bigskip

These results clarify some conditions for which domain of attraction of the mixing distribution is propagated to the associated Poisson mixture distribution.
Very naturally, we denote these situations by two categories of Poisson mixtures: \textbf{Fréchet} and \textbf{Gumbel}.
A broad set of distributions satisfies the 1\textsuperscript{st} Von Mises condition, for example the Fréchet, folded-Cauchy, beta type II, the inverse-gamma or the gamma/beta type II mixture \citep{Irwin}.
Unfortunately, examples are scarce for the Gumbel domain.
Indeed the extra condition on the hazard function is quite restrictive.
Some examples are the lognormal, the Benktander type I and II \citep{Benktander} or the Weibull distribution, with further restrictions on the parameters for the latter two cases.
The Gumbel category does not encompass cases like the Poisson-gamma; while the mixing distribution belongs to $\mathcal{D}_0$, it does not satisfy the additional condition on the hazard function. 
In order to categorise such Poisson mixtures, we study distributions on $\lambda$ that behave like the gamma.

\begin{definition}[\citet{Willmot}]
A density $f$ is a \textbf{gamma type} if 
\begin{equation}\label{gamma_type}
    \lim_{x \to \infty} \frac{f(x)}{C(x)x^{\alpha} e^{-\beta x}} = 1
\end{equation} 
where $C(x)$ is a locally bounded function on $\mathbb{R}_+$ and slowly varying, i.e. $\lim_{t \to \infty} \frac{C(tx)}{C(t)} = 1$ for every $x \in \mathbb{R}_+$ (see \cite{Bingham}), $\alpha \in \mathbb{R}$ and $\beta > 0$.
\end{definition}

 \noindent Using gamma type distributions in the Poisson mixture context allows us to extend the categorisation to cases where $F \in \mathcal{D}_0$, but $F_M \not\in \mathcal{D}_0$.
First, we prove that such a mixing distribution belongs to $\mathcal{D}_0$ (see Proposition \ref{prop1}). Finally Theorem \ref{thm::valiquette} establishes  that $F_M \not\in \mathcal{D}_0$  and quantifies the closeness to the Gumbel domain.\bigskip

\begin{proposition}\label{prop1}
If $F$ is a gamma type distribution, then $F \in \mathcal{D}_0$.
\end{proposition}

\begin{proof}
Let $\overline{F}$ be the survival function of a gamma type distribution. A sufficient condition for $F \in \mathcal{D}_0$ is to show that $F$ has an exponential tail, i.e. $\lim_{x\to \infty} \overline{F}(x+k)/\overline{F}(x) = e^{-\beta k}$ where $k \in \mathbb{R}$ \citep{Shimura}.
    Using L'Hôpital's rule and equation (\ref{gamma_type}), the limit becomes
    $$\lim_{x\to \infty} \frac{\overline{F}(x+k)}{\overline{F}(x)} = e^{-\beta k} \lim_{x \to \infty} \frac{C(x+k)}{C(x)}.$$
    It remains to show that the latter limit is equal to $1$.
    Because $C(\cdot)$ is slowly varying, we can use the Karamata representation
    $$C(x) = c(x) \exp\left( \int_1^x t^{-1} \eta(t) dt \right),$$ where $c(\cdot)$ and $\eta(\cdot)$ are both functions from $\mathbb{R}_+$ to $\mathbb{R}_+$, $\lim_{x \to \infty} c(x) = c > 0$ and $\lim_{x \to \infty} \eta(x) = 0$ \citep{Bingham}.
    Then the limit equals
    $$\lim_{x \to \infty} \frac{C(x+k)}{C(x)} = \lim_{x \to \infty} \exp \left( \int_x^{x+k} t^{-1} \eta(t) dt \right).$$
    Because $\eta(x) \to 0$, for any $\varepsilon > 0$ then $0 < \eta(x) < \varepsilon$ for $x$ large enough. Then $$0 <  \int_x^{x+k} t^{-1} \eta(t) dt < \varepsilon  \int_x^{x+k} t^{-1} dt = \varepsilon \log\left( \frac{x+k}{x} \right) < \varepsilon \log(1 + k)$$ which implies the limit is equal to $1$ and establishes the sufficient condition.
\end{proof}\bigskip

\begin{theorem}
Let $F_M$ be a Poisson mixture with $\lambda$ distributed according to a gamma type distribution $F$. Then for any integer $k \geq 1$, $\lim_{n \to \infty} \frac{1-F_M(n+k)}{1-F_M(n)} = \left(1+\beta \right)^{-k} \in (0,1)$. In particular, $F_M$ is not long tailed ($k=1$).
\label{thm::valiquette}
\end{theorem}

\begin{proof}
Let $P_M$ and $\overline{F}_M$ be the probability and survival functions of a Poisson mixture using a gamma type mixing distribution, then \cite{Willmot} showed that $$\lim_{n \to \infty} \frac{P_M(n)}{C(n) n^\alpha (1+\beta)^{-(n+\alpha +1)}} = 1.$$ 
    Using this result for integer $k$, we obtain
    \begin{align*}
       \lim_{n\to \infty} \frac{\overline{F}_M(n+k+1) - \overline{F}_M(n+k)}{\overline{F}_M(n+1) - \overline{F}_M(n)} &= \lim_{n \to \infty} \frac{P_M(n+k+1)}{P_M(n+1)}\\
        &= \left(\frac{1}{1+\beta}\right)^k \lim_{n \to \infty} \frac{C(n+k+1)}{C(n+1)},
    \end{align*}
    where the last limit converges to $1$ using a similar proof as in Proposition 1. Because $\overline{F}_M$ is monotonically decreasing, the proof can be concluded by applying the Stolz-Cesàro theorem.
\end{proof}\bigskip

\noindent This result allows to characterize a third category: the \textbf{pseudo-Gumbel}. It includes a broad class of mixing distributions among others gamma, gamma/Gompertz, exponential, exponential logarithmic, inverse-Gaussian and its generalization.
\citeauthor{Perline}'s result and  Theorem \ref{thm::valiquette} lead to consider now three categories for Poisson mixtures allowing the clarification of the mixing distribution choices.
For examples, see Table \ref{tab:categories} and the supplementary materials for details.
Additionally, Theorem \ref{thm::valiquette} also quantifies how 'close' those Poisson mixtures are with the quantity $(1 + \beta)^{-1}$ involved in the limit.
Indeed, if $\beta \to 0$, $\frac{1-F_M(n+1)}{1-F_M(n)} \to 1$, i.e. it approaches a long tailed distribution.
Such property can blur the distinction between Gumbel and pseudo-Gumbel for some Poisson mixtures.

\begin{table}[!ht]
\resizebox{\textwidth}{!}{
\begin{tabular}{ |c|c|c| }
 \hline
Mixing ($\lambda$) &  Poisson mixture ($P_M)$ & Category \\ 
 \hline
 Fréchet($a$, $\sigma$)  & Poisson-Fréchet & Fréchet\\
Folded-Cauchy($\mu$, $\sigma$) & Poisson-folded-Cauchy & Fréchet\\
 Inverse-gamma($a$, $b$) & Poisson-inverse-gamma & Fréchet\\
 Beta-II($a$, $b$) & Poisson-beta-II & Fréchet \\
 Gamma/Beta-II-mixture(r,a,b) & Generalized Waring & Fréchet \\
 \hline
 Lognormal($\mu$, $\sigma$)  & Poisson-lognormal & Gumbel\\
 Weibull($a$, $b$) & Poisson-Weibull & Gumbel (if $a < 0.5$)\\
 Benktander-I($a$, $b$) & Poisson-Benktander-I & Gumbel\\
 Benktander-II($a$, $b$) & Poisson-Benktander-II & Gumbel (if $b < 0.5$)\\
 \hline
 Exponential($a$) & Geometric & Pseudo-Gumbel\\
 Gamma($a$, $b$) & Negative binomial & Pseudo-Gumbel\\
 Inverse-Gaussian($\mu$, $\sigma$) & Sichel & Pseudo-Gumbel\\
 Generalized inverse-Gaussian($a$, $b$, $p$) & PGIG & Pseudo-Gumbel\\
 \hline
\end{tabular}}
\caption{Examples of Poisson mixture and associated  categories}
\label{tab:categories}
\end{table}\bigskip

\subsection{Impact of mixing distribution choice on goodness of fit}

Are the categories previously defined  useful to distinguish how the Poisson mixtures behaved? 
How can these categories be efficiently used when it comes to model selection? 
To answer these questions, we simulated 100 samples of different Poisson mixtures with size $n=250$ using the gamma, Fréchet and lognormal distributions on $\lambda$, each one being a representative of the three categories.
For each sample, the Poisson mixture is fitted with the same three distributions and the best model is kept using a Bayesian framework.
This can be done with the \texttt{rstan} package for the language \texttt{R} \citep{Stan} to estimate the parameters by MCMC.
The best model is then kept using the highest \textit{posterior} model probability.
Those probabilities are approximated using the bridge sampling computational technique \citep{Meng} and the dedicated \texttt{R} package \texttt{Bridgesampling} \citep{Bridge}.
All results are based on the following priors: a $\mathrm{gamma}(1,1)$ distribution for positive parameters and a $\mathrm{Normal}(0,1)$ for real parameters. 
Moreover, we simulated for each sample 4 MCMC with 10000 iterations each in order to ensure reasonable convergence for the  parameter estimations and for the \textit{posterior} model probabilities.
Results are presented in Table \ref{tab:} and, in general, the most selected model stood out.
The only exception being the gamma(2,2) where the lognormal and the gamma Poisson mixtures are evenly selected throughout the simulations. 

\begin{table}[!ht]
\centering
\begin{tabular}{ |c||c|c|c| }
 \hline
Poisson mixture &  Fréchet & Lognormal & Gamma \\
\hline
Fréchet(1,1) & \bf{95} & 5 & 0\\
Fréchet(2,1) & \bf{77} & 17 & 6\\
\hline
Lognormal(0,1) & 16 & \bf{63} & 21\\
Lognormal(1,1) & 6 & \bf{86} & 8\\
\hline
Gamma(2,1) & 0 & 23 & \bf{77}\\
Gamma(2,2) & 1 & 47 & \bf{52}\\ 
\hline
\end{tabular}
\caption{Selected model frequencies for each Poisson mixture simulation with the highest frequency in bold.}
\label{tab:}
\end{table}\bigskip

\pagebreak

This example shows the importance in the comparison of various mixing distributions.
However, this approach based on systematic comparisons may suffer from computational limitations and is difficult to used in practice where a selection objective involves too much possibilities for the mixing distribution.
Indeed, such an approach requires an appropriate choice of priors, a high number of MCMC iterations, and a study of convergence for each mixing distribution.
Here, we assume the same priors and the same number of iterations for each mixing distribution in order to systematically simulate all our categories.
In reality, each case must be studied with care and such an approach depends too much on the latent variables of the mixture model.
As an alternative, we propose a simple strategy that uses directly the data and allows the user to focus on a specific family of mixing distributions.
In doing so, the estimation of the latent variable can be done as a last step.
The proposed alternative (see Section \ref{Section_Strategy}) relies on the use of a sequential approach by selecting first the most appropriate category (Gumbel, Fréchet or pseudo-Gumbel) and next by comparing only a few representative distributions belonging to this selected domain. For instance, such representative distributions are lognormal/Benktander-I for the Gumbel category, inverse-gamma/folded-Cauchy for the Fréchet category or a gamma/inverse-Gaussian for the pseudo-Gumbel case.

\section{Strategy}\label{Section_Strategy}

This section proposes a strategy to choose a mixing distribution on $\lambda$ using the categories defined in Section \ref{Section_Mixture}.
As previously mentioned, an excess of zeroes and extreme values create overdispersion in count data which induces a particular tail behavior. 
The main idea is to choose mixing distributions among the three categories reflecting which ones best fit the empirical tail behavior.
Peaks-over-threshold (POT) method \citep{Coles} is well adapted for this purpose.
This technique analyses the distribution of the excesses defined by $Y - u|Y > u$. 
\cite{Pickands}, \cite{Balkema} showed that $Y$ belongs to a domain of attraction if and only if the distribution of the excesses can be uniformly approached by a generalized Pareto distribution (GPD) as $u$ tends to the right endpoint of the distribution of $Y$. The corresponding cdf is given by
\begin{equation}
 {H}_{\gamma, \sigma}(y) =
    \begin{cases}
     1- \left(1 + \gamma \frac{y}{\sigma} \right)^{-1/\gamma} & \text{if $\gamma \neq 0$}\\
     1- \exp\left( -\frac{y}{\sigma}  \right) & \text{otherwise}
    \end{cases}       
\end{equation}
with support $\mathbb{R}_+$ if $\gamma \geq 0$ or $\left[0; -\frac{\sigma}{\gamma} \right]$ if $\gamma < 0$ and where $\gamma \in \mathbb{R}$ and $\sigma > 0$ are respectively shape and scale parameters. The sign of the
$\gamma$ parameter is intrinsically related to the domain of attraction.
Indeed, the distribution of $Y$ belongs to $\mathcal{D}_-$, $\mathcal{D}_0$ or $\mathcal{D}_+$ if $\gamma < 0$, $\gamma = 0$ or $\gamma > 0$ respectively.
Therefore, fitting a GPD to the excesses of the count data can inform  whether or not Poisson mixture distribution belongs to a known domain of attraction and, if so, which one.

\subsection{Decision tree}

Suppose  overdispersed count data for which we need to fit a Poisson mixture. Our strategy to select appropriate mixing distributions is based on a decision tree (Figure \ref{fig:tree})  leading to the three categories defined in Section \ref{Section_Mixture}.
The first step consists in the selection of a threshold $u$ large enough for the data and to fit a GPD to the excesses.
The choice of $u$ can either be based on empirical quantile or on studying the mean residual life plot \citep{Threshold} and the GPD parameters can be efficiently estimated using maximum likelihood \citep{Coles}.\bigskip

\begin{figure}[H]
    \centering
    \includegraphics[width=\textwidth]{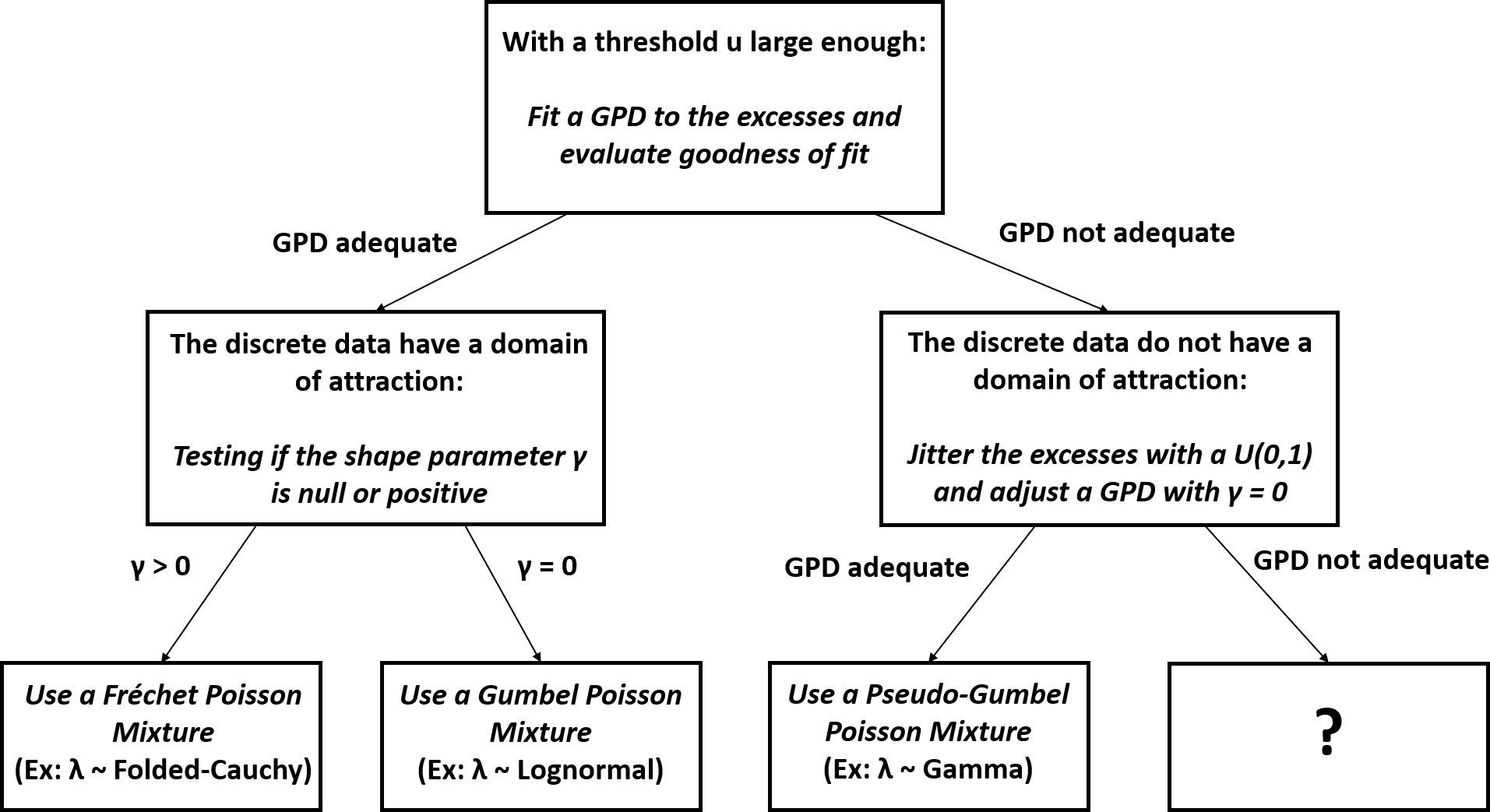}
    \caption{Decision tree for Poisson mixtures}
    \label{fig:tree}
\end{figure}

Two situations arise: first, if excesses are correctly fitted using the GPD model (left side in Figure~\ref{fig:tree}), then we propose to use a  distribution belonging to Gumbel or Fréchet such that the resulting Poisson mixture distribution remains in these domains.
Indeed, if the excesses can be approximated by a GPD, then the Poisson mixture must be in a domain of attraction (see \cite{Pickands}, \cite{Balkema}).
Therefore, the categories defined by \cite{Perline} are an adequate set of mixing distributions to choose from.
The choice of the specific category is directly based on the estimation of the shape parameter obtained at the first step.
Testing whether $\gamma = 0$ or not can be done with the deviance statistic \citep{Coles}.
If it is the case, we should use a Poisson mixture in the Gumbel category.
Else if $\gamma$ is positive, the Fréchet category should be prioritize. Otherwise, if $\gamma$ is negative, this strategy cannot assess which mixing distribution should be used.
However, no such case has been relevant in our study.
\bigskip

The second situation (right side in Figure~\ref{fig:tree}) corresponds to the case where GPD model is not well adapted. In this case, any mixing distributions such that the Poisson mixture belongs to Gumbel or Fréchet domains of attraction should be avoided. A distribution from the pseudo-Gumbel category should potentially be favoured.
As demonstrated by \citet{Shimura}, such discrete random variables originate from an unique continuous distribution in $\mathcal{D}_0$ that has been discretized.
That is why we transform the excesses to continuous values thanks to a \textit{jittering} technique consisting in the addition of a continuous random noise \citep{Nagler}. 
For an application, \cite{Coeurjolly-Rousseau} used this technique to study the Poisson's median and to construct an estimator for $\lambda$.
In our case, the excesses $Y-u|Y > u$ are discrete and greater or equal to $1$.
However the GPD with $\gamma \geq 0$ is defined on $(0, \infty)$.
To transform the excesses into the same support, we jittered by subtracting a $\mathrm{Uniform(0,1)}$.
With these ``jittered`` data points, we fit a GPD again by fixing $\gamma = 0$ and testing again if it is adequate.
If, in this case, the fit is adequate, we consider that the data are pseudo-Gumbel. 
Otherwise, one should proceed to another approach in order to choose a mixing distribution.
However, we rarely encountered this situation in our simulations.

\subsection{Evaluation of a sequential approach for mixing distribution selection}

In order to study the performance of the proposed strategy, various Poisson mixtures samples have been simulated.
The decision tree is then systematically applied using
the \texttt{evd} package \citep{EVD} for maximum likelihood estimation of GPD parameters, the modified Anderson Darling test for the goodness-of-fit and deviance statistic for testing the nullity of the shape parameter.

Let $X_1, \dots, X_m$ denote $m$ i.i.d random variables ordered as $X_{(1)} \leq \dots \leq X_{(m)}$.
The modified Anderson-Darling test statistic for a distribution $H$ is defined by
$$T(X_1, \dots, X_m) = \frac{n}{2} - 2\sum_{i=1}^m H(X_{(i)}) - \sum_{i=1}^m\left[ 2 - \frac{2i - 1}{n} \right] \log(1-H(X_{(i)})). $$ 
As presented in \cite{GPD_Review}, this statistic has an asymptotic distribution defined by a weighted sum of $\chi_1^2$.
For this simulation study, the $m$ random variables are the excesses and $H$ is the GPD.
We point that such a test works for any distribution $H$.
However, some tests exist specifically when $H$ is a GPD.
See \cite{Toulemonde} or \cite{Villasenor-Gonzalez} for examples.

\pagebreak

Finally, to test $H_0:\gamma = 0$ versus $H_1:\gamma\neq 0$, we fit the two models, that is the complete one and the restricted one, evaluate the corresponding  log likelihoods namely $\mathcal{L}_1$ and $\mathcal{L}_0$ and we conclude with  the deviance statistic $D = 2\left(\mathcal{L}_1 - \mathcal{L}_0 \right)$ which follows a $\chi_1^2$ under suitable conditions \citep{Coles}. \\

The following simulation scheme is then applied in the language \texttt{R}:

\begin{enumerate}
    \item For a fixed sample size $n$ and a Poisson mixture $F_M$ with fixed parameters, simulate the mixed Poisson observations $\mathbf{Y} :=\left(Y_1, \dots, Y_n\right)$.
    \item For a threshold $u$ based on the sample $\mathbf{Y}$ (example: $95\mathrm{th}$ quantile), get the excesses $\mathbf{X} := \mathbf{Y} - u | \mathbf{Y} > u$.
    \item Calculate the MLE of $\gamma$ and $\sigma$ of the GPD for $\mathbf{X}$ using \texttt{evd::fpot} function using the Nelder-Mead optimization method. 
    \item Test the GPD for $\mathbf{X}$ with the modified Anderson Darling test ($\alpha = 0.05$). 
    The p-values are calculated with a bootstrap approach using $250$ iterations (see \cite{GPD_Review}).
    \item Evaluate which category the sample is classified to using the decision tree with the following outcomes:
    \begin{enumerate}
        \item  If the test at step 4 for $\mathbf{X}$ is not rejected, use the deviance statistic to find which domain of attraction $\mathbf{X}$ belongs to. 
        If $\gamma < 0$ is significant, the sample fails to have a category.
        \item Else repeat steps 3 and 4 for the jittered excesses $\mathbf{X}^c := \mathbf{X} - \mathrm{Unif}(0,1)$ and by fixing $\gamma = 0$.
        If the GPD is not rejected, the sample belongs to pseudo-Gumbel category. 
        Otherwise the sample fails to have a category.
    \end{enumerate}
    \item Repeat 1000 times the steps 1 to 5.
\end{enumerate}

Distributions from the three categories are tested using these steps with $n$ equal to $1000$ or $2000$ and the threshold fixed at the $95 \mathrm{th}$ and $97.5 \mathrm{th}$ empirical quantiles for $n = 1000$ and  we add the $98.5\mathrm{th}$ empirical quantile for $n = 2000$.
For the Fréchet category, Fréchet and folded-Cauchy mixing distributions are simulated.
For the Gumbel category, lognormal and Weibull mixing distributions have been simulated.
Finally, gamma and inverse-Gaussian mixing distributions are simulated for the pseudo-Gumbel.
Results are presented in Table 3 and also in the Supplementary Materials section (other sets of parameters, different sample sizes and different thresholds).

\begin{table}[H]
\centering
\resizebox{0.8\textwidth}{!}{
\begin{tabular}{|c|c|c|c|}
\hline
 Mixing distribution & Average excesses & GPD Rejection & Category \\
\hline
Fréchet(1,1) & 48.727 & 0.069 & 0.917 (Fréchet)\\
Folded-Cauchy(0,1) & 48.243 & 0.078 & 0.896 (Fréchet)\\
\hline
Lognormal(1,1) & 46.750 & 0.126 & 0.720 (Gumbel)\\
Weibull(0.5, 1) & 46.246 & 0.133 & 0.754 (Gumbel)\\
\hline
Gamma(2,1) & 36.200 & 0.704 & 0.635 (Pseudo-Gumbel)\\
Inverse-Gaussian(1,2) & 38.977 & 0.856 & 0.709 (Pseudo-Gumbel)\\
\hline
\end{tabular}
}
\label{tab:sim}
\caption{Average number of excesses, sample proportion of GPD rejection and of the most frequent category for the simulations with $n = 1000$ and $u = 95\mathrm{th}$ empirical quantile}
\end{table}

These simulations aim to assess whether the decision tree adequately identifies the Poisson mixture categories.
To do so, we calculate the proportion of samples where the GPD is rejected in the first branch and the proportion of the most frequent category where the Poisson mixture is identified.
For cases like the Gumbel and Fréchet categories, we should see a low GPD rejection frequency.
Conversely, the pseudo-Gumbel category should have a high GPD rejection.
For all cases, we should have a high proportion of samples adequately identified to their category.\bigskip

As presented in Table 3, most of the Poisson mixtures simulated can be adequately identified to the appropriate category. 
For the mixing distributions Fréchet(1,1), folded-Cauchy(0,1), lognormal(1,1) and Weibull(0.5,1), the GPD is mostly adequate for the excesses due to its low rejection rate.
For pseudo-Gumbel mixtures Gamma(2,1) and inverse-Gaussian(1,2), the GPD is mostly rejected and, once the excesses are jittered, the Gumbel domain is found.
Moreover, the simulation results considering a Weibull($a$,$b$) as a mixing distribution are relevant with the theory.
Indeed, as mentioned in Table 1, the Poisson-Weibull is in the Gumbel category if its parameter $a$ is smaller than $0.5$.
In comparison to the limit case Weibull(0.5, 1) in Table 3, we also tested the Weibull(1,1) and found that this mixture does not belong to Gumbel category (Table 2 Supp. Material).\bigskip

Some interesting factors have been identified concerning the categorization.
First, as noted by \cite{Hitz}, the discrete excesses need a certain amount of variability in order to have a smooth adjustment to the GPD.
If the variance is not high enough, it can be difficult in practice to adequately identify the mixture category.
For example, based on the simulations, the Fréchet(2,1) Poisson mixture belongs to the Fréchet domain only after the excesses are jittered (Table 1 Supp. Material).
This distribution has its expectation defined which results in a more stable sample compared to the Fréchet(1,1). 
Here, both distributions do not have a defined variance because their parameter $a \leq 2$.
Still, because the Fréchet(2,1) is the limit case, it leads to a less volatile sample of excesses.
This situation is identified in the lognormal case as well.
Indeed, the Poisson mixture using the lognormal(0,1) has a smaller variance compare to mixture with a lognormal(1,1).
In our simulations, the former case has some difficulty to be adequately identified to the Gumbel category compared to the latter (Table 2 Supp. Material).\bigskip

Another important aspect is the choice of the threshold $u$.
Indeed, the threshold affects the variance and the domain of attraction inferred by the data.
For example, the lognormal(0,1) may have difficulties to be identified due to its variance, but once the $u$ is large enough we do find the Gumbel category.
For instance when $n = 1000$ and $u = 95\mathrm{th}$ quantile, 257 over 1000 samples are classified to the Gumbel category compared to 476 over 1000 samples in the pseudo-Gumbel category.  
However, for the same $n$ and $u = 97.5\mathrm{th}$, 790 over 1000 samples are classified to the Gumbel category compared to 64 over 1000 samples in the pseudo-Gumbel category.
Therefore a larger threshold $u$ is necessary in this case. 
However, it can be too large for some distributions like the gamma.
In particular, when $n = 1000$ and $u = 95\mathrm{th}$ quantile there should be in average $50$ excesses, but Table 3 indicates that the gamma(2,1) has $36.2$ excesses.
This can be explained by the lack of different values in the right tail of the Poisson mixture which leads to an underrepresented sample of excesses.
Such discrepancy greatly affect the categorisation.
For example, with $n = 2000$ and $u = 98.5\mathrm{th}$ quantile, the gamma(2,1) has in average $21.927$ excesses and are mostly classified in the Gumbel domain (Table 3 Supp. Material).
Clearly this  inappropriate classification is due to the lack of excesses.\bigskip

Finally, some pseudo-Gumbel Poisson mixtures can be very close to the Gumbel domain. 
For instance, the inverse-Gaussian$(1,2)$ is adequately classified, but the inverse-Gaussian$(2,1)$ is mostly classified in the Gumbel category (Table 3 Supp. Material).
This can be explained by the 'closeness' property described in Theorem 1.
Indeed, the density of a inverse-Gaussian$(\mu,\sigma)$ can be represented by

$$f(x) = C(x) x^{-3/2} \exp\left(-\frac{\sigma}{2\mu^2} x\right)$$
where $C(x)= C \exp\left(-\frac{\sigma}{2x}\right)$ with $C$ the normalizing constant.
By equation (3), $\beta = \sigma/2\mu^2$ and substituting the values $\mu = 2$ and $\sigma = 1$ gives $\beta = 1/8$. Therefore the limit defined in Theorem 1 indicates that the resulting Poisson mixture has
$$\lim_{n\to \infty} \frac{1-F_M(n+1)}{1-F_M(n)} = \frac{8}{9}.$$
Because this limit is pretty close to that of a long tailed distribution, i.e. the limit is near to $1$, the distinction between pseudo-Gumbel and Gumbel gets blurred.
To visualize how this 'closeness' affects the fit, let the parameter $\mu = 2$ be fixed and vary $\sigma$ from $0.1$ to $8$ for the inverse-Gaussian.
For each value of $\sigma$, simulate 500 samples of size $n=2000$ from the Poisson mixture, fix the threshold $u$ to the $97.5 \mathrm{th}$ empirical quantile and calculate the proportion of samples where the GPD is rejected with p-value $\alpha = 0.05$.
In theory, the limit from Theorem 1 should approach $0$ when $\sigma$ gets larger, which implies getting further from the Gumbel domain and results into more rejection of the GPD. 
This result is reflected in Figure \ref{fig:reject} when $\sigma$ approaches $8$.

\begin{figure}[H]
    \centering
    \includegraphics[width=0.65\textwidth]{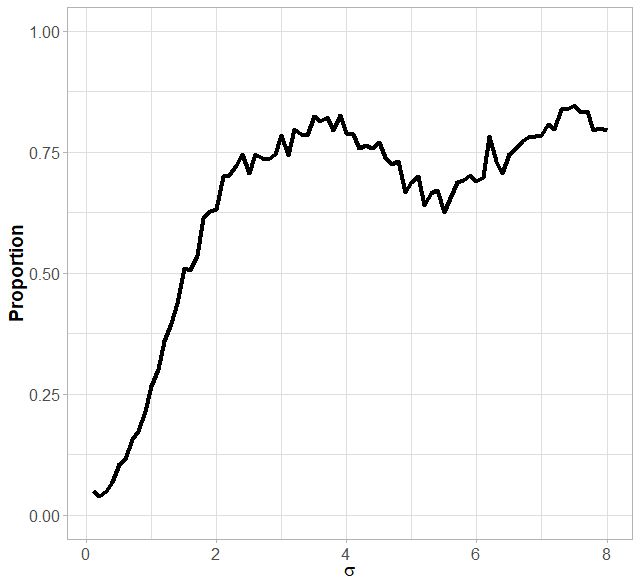}
    \caption{Proportion of inverse-Gaussian($2$, $\sigma$) Poisson mixture samples (size $n=2000$) where the GPD has been rejected ($\alpha = 0.05$) for the excesses ($u = 97.5 \mathrm{th}$ quantile) according to $\sigma$.}
    \label{fig:reject}
\end{figure}

\section{Conclusion and perspectives}
Overdispersed count data are commonly observed in many applied fields and Poisson mixtures are appealing to model such data \citep{Karlis}. 
However, the choice of the appropriate mixing distribution is a difficult task relying mainly on empirical approaches related to modelers subjectivity or on intensive computational techniques combined with goodness-of-fit test or information criteria. 
In this paper, we proposed a new strategy based on the analysis of the tail behavior of the data.
We extend the usual Gumbel and Fréchet domains of attraction introducing the pseudo-Gumbel category for Poisson count data.
In particular, we show how tail behavior can provide a great amount of information to evaluate the mixing distributions.
Based on a sequential strategy and decision tree, we proposed a useful and efficient approach to select the most appropriate category allowing to focus on a more restrictive set of potential candidates.
The choice of the most appropriate distribution within a given category is not dealt with in this paper.
Some strategies can be proposed helping the choice of such potential candidates.
More specifically, it could be based on the simplicity either for the inferential step or for the inclusion of covariates or yet for biological interpretations.
Moreover, recently, tremendous researches have been developed to jointly model count data.
For instance, the joint species distribution models are proposed extending classical species distribution models in ecology \citep{JSDM} and are often based on the use of the multivariate lognormal distribution \citep{Aitchison,Robin}.
Based on our approach, various and flexible models could be developed combining different mixing distribution belonging to different categories (Gumbel, Fréchet or pseudo-Gumbel) and the use of copulas to model dependencies structures between continuous mixing distributions \citep{Nelsen}.

\section*{Acknowledgments}

This research was supported by the GAMBAS project funded by the French National Research Agency (ANR-18-CE02-0025) and the french national programme LEFE/INSU.
We also thank Éric Marchand for the helpful comments and fruitful discussions.

\bibliography{biblio}

\end{document}